\documentclass[reqno]{amsart}
\usepackage{amsmath}

\usepackage{amsfonts}
\usepackage{amssymb}
\usepackage{hyperref}

\usepackage[a4paper,top=3.5cm,bottom=3cm,left=3.5cm,right=3.5cm]{geometry}

\begin{document}
\title[Positive solutions for a fourth-order boundary value problem]{Existence and multiplicity of positive solutions to a fourth-order multi-point boundary value problem}
\author[F. Haddouchi]{Faouzi Haddouchi}
\address{
Department of Physics, University of Sciences and Technology of
Oran-MB, El Mnaouar, BP 1505, 31000 Oran, Algeria
\newline
And
\newline
Laboratoire de Math\'ematiques Fondamentales et Appliqu\'ees d'Oran (LMFAO). Universit\'e Oran1. B.P. 1524 El Mnaouer, Oran, Alg\'erie.}
\author[C. Guendouz]{Cheikh Guendouz}
\address{
Laboratoire de Math\'ematiques Fondamentales et Appliqu\'ees d'Oran (LMFAO), D\'epartement de Math\'ematiques, Universit\'e Oran1. B.P. 1524 El Mnaouer, Oran, Alg\'erie.}
\author[S. Benaicha]{Slimane Benaicha}
\address{
Laboratoire de Math\'ematiques Fondamentales et Appliqu\'ees d'Oran (LMFAO), D\'epartement de Math\'ematiques, Universit\'e Oran1. B.P. 1524 El Mnaouer, Oran, Alg\'erie.}

\email{fhaddouchi@gmail.com}
\email{guendouzmath@yahoo.fr}
\email{slimanebenaicha@yahoo.fr}

\begin{abstract}
In this paper, we study the existence and multiplicity of positive solutions for a nonlinear fourth-order with  multi-point  boundary  conditions involving an integral boundary condition. The main tool is Krasnosel'skii fixed point theorem on cones.
\end{abstract}

\maketitle \numberwithin{equation}{section}
\newtheorem{theorem}{Theorem}[section]
\newtheorem{lemma}[theorem]{Lemma}
\newtheorem{definition}[theorem]{Definition}
\newtheorem{proposition}[theorem]{Proposition}
\newtheorem{corollary}[theorem]{Corollary}
\newtheorem{remark}[theorem]{Remark}
\newtheorem{exmp}{Example}[section]

\section{Introduction\label{sec:1}}
Boundary value problems related to nonlocal conditions have many applications in many problems such as in the theory of heat conduction, thermoelasticity, plasma physics, control theory, etc. The current analysis of these problems has a great interest and many methods are used to solve such problems. Recently, the study of existence of positive solution to fourth-order boundary value problems has gained much attention and is rapidly growing field, see \cite{And, Gra, Ben,Han, Ma,y sun,Kang,Zou,Li,Webb}. However, the approaches used in the literature are usually topological degree theory and fixed-point theorems in cone \cite{Kras}.

 Multi-point boundary value problems have received considerable interest in the mathematical applications in different areas of science and engineering \cite{BO YANG1,Zhang2,Bai,Zhang}.\\

 In 2007, M. Zhang and Z. Wei \cite{Zhang},  studied the existence of multiple positive solutions for fourth-order $m$-point boundary value problem
       \begin{equation}\label{eq001234}
       \begin{cases}u^{(4)}(t) + B(t)u^{\prime \prime}-A(t)u= f(t,u), \ 0< t<1, \\
      u(0)= \sum_{i=1}^{m-2} a_{i}u(\xi_{i}),\ u(1)=\sum_{i=1}^{m-2} b_{i} u(\xi_{i}),\\
      u^{\prime \prime}(0)= \sum_{i=1}^{m-2} a_{i} u^{\prime \prime}(\xi_{i}),\ u^{\prime \prime}(1)= \sum_{i=1}^{m-2} b_{i} u^{\prime \prime}(\xi_{i}).
       \end{cases}
       \end{equation}
 And in the same year, X. Zhang and L. Liu  \cite{Zhang2}, considered the fourth-order multi-point boundary value problems with bending term
       \begin{equation}\label{eq0012345}
              \begin{cases} x^{(4)}(t) = g(t) f(t,x(t),x^{\prime \prime}(t)) , \ t \in (0,1),\\
             x(0)= 0 ,\ x(1)= \sum_{i=1}^{m-2} a_{i} x(\xi_{i}),\ x^{\prime \prime}(0)= 0 ,\
             x^{\prime \prime}(1)= \sum_{i=1}^{m-2} b_{i} x^{\prime \prime}(\xi_{i}).
              \end{cases}
              \end{equation}

 In 2016, S. Benaicha and F. Haddouchi \cite{Ben}, considered the following  fourth-order two-point boundary value problem
  \begin{equation}\label{eq0003}
    u^{\prime \prime \prime \prime}(t) + f(u(t)) = 0,\  t \in (0,1),
    \end{equation}
    \begin{equation}\label{eq0004}
    u^{\prime}(0) = u^{\prime} (1) =u^{\prime \prime}(0) = 0,\  u(0)= \int_{0}^{1} a(s) u(s) ds.
    \end{equation}

  In 2017, Bo Yang \cite{BO YANG1}, studied the fourth-order differential equation
            \begin{equation}\label{eq000351}
               u^{\prime \prime \prime \prime}(t) =g(t) f(u(t)),\  t \in (0,1),
               \end{equation}
           together with boundary conditions
            \begin{equation}\label{eq000352}
              u(0) = \alpha u^{\prime}(0)-\beta u^{\prime \prime}(0) = \gamma u^{\prime}(1) + \delta u^{\prime \prime}(1) =  u^{\prime \prime \prime}(1) = 0.
               \end{equation}

 In 2018, Yan. D and R. Ma \cite{Yan}, investigated the global behavior of positive solutions of fourth-order
boundary value problems
 \begin{equation}\label{eq000353}
               u^{\prime \prime \prime \prime} ={\lambda}f(x, u),\  x \in (0,1),
               \end{equation}
           together with boundary conditions
            \begin{equation}\label{eq000354}
              u(0) =u(1)= u^{\prime \prime}(0) = u^{\prime \prime}(1)= 0.
               \end{equation}
where $f : [0, 1] \times {\mathbb{R}}^{+} \rightarrow \mathbb{R}$ is a continuous function with $f(x, 0) < 0$ in $(0, 1)$, and $\lambda> 0$. The proof of main results are based upon bifurcation techniques.

 In 2019, Wei. Y et al. \cite{Wei}, considered the following boundary value problem
 \begin{equation}\label{eq000353}
               u^{(4)}(t) =f(t,u(t),u^{\prime}(t)),\  t \in (0,1),
               \end{equation}
           subject to the boundary conditions
            \begin{equation}\label{eq000354}
              u(0) = u^{\prime}(0)=u^{\prime}(1) =u^{\prime \prime}(1)=0.
               \end{equation}

Under some conditions of $f$, the existence and uniqueness of this problem is obtained.

For some other results on boundary value problems, we refer the reader to the papers \cite{Gra,Han,Ma2,shen,BO YANG2,Zhang1,Lv}.

 Motivated by these works, in this paper, we are concerned with the following fourth-order multi-point with integral boundary condition
 \begin{equation}\label{eq001}
  u^{\prime \prime \prime \prime}(t) + f(t,u(t)) = 0,\  t \in (0,1),
  \end{equation}
  \begin{equation}\label{eq002}
  u^{\prime} (0) = u^{\prime} (1) =u^{\prime \prime}(0) = 0,\ u(0) = \alpha \int_{0}^{1}u(s)ds + \sum_{i=1}^{n} \beta_{i} u(\eta_{i}),
  \end{equation}
  where
  \begin{itemize}
 \item[(C1)] $f \in C([0,1]\times [0,\infty) ,[0,\infty));$
\item[(C2)] $\alpha\geq 0,\beta_{i}\geq 0,1\leq i \leq n $\ \text{and} \ $0 < \eta_{1}<\eta_{2} <... <\eta_{n}< 1;$
\item[(C3)] $\alpha + \sum_{i=1}^{n} \beta_{i}  < 1.$
  \end{itemize}

 This paper is organized as follows. In section 2, we present some theorems and lemmas that will be used to prove our main results. In section 3, we discuss the existence of at least one positive solution for  \eqref{eq001}-\eqref{eq002}. In section 4, we investigate the existence of multiple positive solutions for \eqref{eq001}-\eqref{eq002}. Finally, we give some examples to illustrate our results in section 5.

 \section{Preliminaries}
 At first, we consider the Banach space $C([0,1],\mathbb{R})$  equipped with the sup norm \[\|u\|=\\sup_{t\in[0, 1]}|u(t)|.\]
 \begin{definition}
 Let $E$ be a real Banach space. A nonempty, closed, convex set $
 K\subset E$ is a cone if it satisfies the following two conditions:
 \begin{itemize}
 \item[(i)]
  $x\in K$, $\lambda \geq 0$ imply $\lambda x\in K$;
 \item[(ii)]
 $x\in K$, $-x\in K$ imply $x=0$.
 \end{itemize}
 \end{definition}

 \begin{definition}
 An operator $T:E\rightarrow E$ \ is completely continuous if it is continuous
 and maps bounded sets into relatively compact sets.
 \end{definition}
 \begin{definition}
 $ A $ function $ u(t)$ is called a positive solution of \eqref{eq001} and \eqref{eq002}
 if $ u \in C ([0,1]) $ and $ u(t) > 0 $ for all $ t \in(0,1).$
 \end{definition}

 To prove our results, we need the following well-known fixed point theorem
 of cone expansion and compression of norm type due to Krasnoselskii \cite{Kras}.

 \begin{theorem}\label{T1}
 Let $E$ be a Banach space, and let $K\subset E$, be a cone. Assume that $%
 \Omega_{1}$ and $\Omega_{2}$ are bounded open subsets of $E$ with $0\in \Omega _{1}$,
 $\overline{{\Omega }}_{1}\subset \Omega_{2}$ and let
 \[
 A:K\cap  ( \overline{{\Omega}}_{2} \backslash \Omega_{1} )\rightarrow K
 \]
 be a completely continuous operator such that
 \begin{itemize}
 \item[(a)]
 $\left\Vert Au\right\Vert \leq \left\Vert u\right\Vert ,$ $u\in K\cap
 \partial
 \Omega _{1}$, and $\left\Vert Au\right\Vert \geq \left\Vert u\right\Vert ,$
 $u\in K\cap \partial \Omega_{2}$; or
 \item[(b)]
 $\left\Vert Au\right\Vert \geq \left\Vert u\right\Vert ,$ $u\in K\cap
 \partial
 \Omega_{1}$, and  $\left\Vert Au\right\Vert \leq \left\Vert u\right\Vert ,$
 $u\in K\cap \partial \Omega_{2}.$
 \end{itemize}
 Then $A$ has a fixed point in $K\cap  ( \overline{{\Omega }}_{2} \backslash \Omega_{1} )$.
 \end{theorem}

 Let the multi-point boundary value problem
 \begin{equation}\label{eq1}
 u^{\prime \prime \prime \prime}(t) + y(t) = 0,\  t \in (0,1),
 \end{equation}
 \begin{equation}\label{eq2}
 u^{\prime} (0) = u^{\prime} (1) =u^{\prime \prime}(0) = 0,\ u(0) = \alpha \int_{0}^{1}u(s)ds + \sum_{i=1}^{n} \beta_{i} u(\eta_{i}).
 \end{equation}

For convenience, we denote $ k= 1-\bigg(\alpha + \sum_{i=1}^{n} \beta_{i}\bigg).$
\begin{lemma}\label{l1}
Let $k \neq 0 $. Then for any $ y \in C[0,1] $, the boundary value problem \eqref{eq1}-\eqref{eq2} has a unique solution which can be expressed by
$$ u(t) = \int_{0}^{1}  H(t,s)y(s) ds ,$$
where $ H(t,s) : [0,1] \times [0,1]\rightarrow \mathbb{R}$ is the Green's function defined by
\begin{equation}\label{eq51}
H(t,s)= G(t,s) +\frac{\alpha}{k} \int_{0}^{1}G(\tau ,s ) d\tau +\frac{1}{k} \sum_{i=1}^{n} \beta_{i} G( \eta_{i} ,s ) ,
\end{equation}
and
\begin{equation}\label{eq5}
G(t,s)=\frac{1}{6} \begin{cases}t^{3}(1-s)^{2}-(t-s)^{3}, & 0 \leq s \leq t \leq 1; \\
t^{3}(1-s)^2 , & 0 \leq t \leq s \leq 1.
\end{cases}
\end{equation}
\end{lemma}

\begin{proof}
Rewriting \eqref{eq1} as $ u^{\prime \prime \prime \prime}(t) = -y(t) $ and integrating three times over the interval $ [0, t]$ for $ t \in [0,1] $, we obtain
$$ u^{\prime \prime \prime}(t) = - \int_{0}^{t} y(s) ds + C_{1}, $$
$$ u^{\prime \prime}(t) = - \int_{0}^{t} (t-s)y(s) ds + C_{1} t + C_{2}, $$
$$ u^{\prime}(t) = -\frac{1}{2} \int_{0}^{t}( t- s)^{2} y(s) ds +\frac{1}{2} C_{1} t^{2} + C_{2} t + C_{3}, $$
\begin{equation} \label{eq363}
u(t) = -\frac{1}{6} \int_{0}^{t}( t- s)^{3} y(s) ds +\frac{1}{6} C_{1} t^{3} +\frac{1}{2} C_{2} t^{2} + C_{3} t + C_{4},
\end{equation}
where $C_{1}, C_{2}, C_{3}, C_{4} \in \mathbb{R}$ are constants.
By \eqref{eq2}, we get \\
\[C_{2} =C_{3}= 0\ \text{and}\ C_{1} = \int_{0}^{1} (1-s)^{2} y(s) ds.\]
Further,
\begin{eqnarray*}
C_{4}& =&u(0)\\
 & =& \alpha  \int_{0}^{1} \bigg( -\frac{1}{6} \int_{0}^{\tau} (\tau-s)^{3}y(s) ds\\
 &&+\frac{\tau^{3}}{6} \int_{0}^{1} (1-s)^{2} y(s) ds + C_{4}  \bigg) d\tau \\
 &&+\sum_{i=1}^{n} \beta_{i} \bigg(-\frac{1}{6} \int_{0}^{\eta_{i}} (\eta_{i} -s)^{3} y(s) ds + \frac{\eta_{i}^{3}}{6} \int_{0}^{1} (1-s)^{2} y(s) ds + C_{4}   \bigg)\\
& =& \alpha \int_{0}^{1} \bigg( -\frac{1}{6} \int_{0}^{\tau} (\tau-s)^{3}y(s) ds + \frac{\tau^{3}}{6} \int_{0}^{1} (1-s)^{2} y(s) ds \bigg) d\tau \\
 &&+\sum_{i=1}^{n} \beta_{i} \bigg(-\frac{1}{6} \int_{0}^{\eta_{i}} (\eta_{i} -s)^{3} y(s) ds + \frac{\eta_{i}^{3}}{6} \int_{0}^{1} (1-s)^{2} y(s) ds   \bigg)\\
  &&+ C_{4} \bigg( \alpha +\sum_{i=1}^{n} \beta_{i} \bigg),
\end{eqnarray*}
so
\begin{equation}
\begin{split}
C_{4}= & \frac{\alpha}{k}  \int_{0}^{1} \bigg( -\frac{1}{6} \int_{0}^{\tau} (\tau-s)^{3}y(s) ds + \frac{\tau^{3}}{6} \int_{0}^{1} (1-s)^{2} y(s) ds  \bigg) d\tau \\
 & + \frac{1}{k} \sum_{i=1}^{n} \beta_{i} \bigg(-\frac{1}{6} \int_{0}^{\eta_{i}} (\eta_{i} -s)^{3} y(s) ds + \frac{\eta_{i}^{3}}{6} \int_{0}^{1} (1-s)^{2} y(s) ds \bigg).
\end{split}
\end{equation}
Replacing these expressions in \eqref{eq363}, we get

\begin{eqnarray*}
u(t)& =&- \frac{1}{6} \int_{0}^{t} (t-s)^{3} y(s) ds + \frac{t^{3}}{6} \int_{0}^{1} (1-s)^{2} y(s) ds \\
 && + \frac{\alpha}{k} \int_{0}^{1} \bigg( -\frac{1}{6} \int_{0}^{\tau} (\tau-s)^{3}y(s) ds + \frac{\tau^{3}}{6} \int_{0}^{1} (1-s)^{2} y(s) ds  \bigg) d\tau \\
 && + \frac{1}{k} \sum_{i=1}^{n} \beta_{i} \bigg(-\frac{1}{6} \int_{0}^{\eta_{i}} (\eta_{i} -s)^{3} y(s) ds + \frac{\eta_{i}^{3}}{6} \int_{0}^{1} (1-s)^{2} y(s) ds   \bigg) \\
 &=& \frac{1}{6} \int_{0}^{t} [ t^{3}(1-s)^{2} - (t-s)^{3} ]  y(s) ds +\frac{1}{6} \int_{t}^{1} t^{3} (1-s)^{2} y(s) ds \\
 && + \frac{\alpha}{6k} \int_{0}^{1} \bigg( \int_{0}^{\tau}[ \tau^{3} (1-s)^{2} - (\tau-s)^{3} ] y(s) ds
  +\int_{\tau}^{1} \tau^{3} (1-s)^{2} y(s) ds  \bigg) d\tau \\
 && + \frac{1}{6k}  \sum_{i=1}^{n} \beta_{i} \bigg( \int_{0}^{\eta_{i}} [ \eta_{i}^{3} (1-s)^{2} - (\eta_{i} -s)^{3} ]  y(s) ds + \eta_{i}^{3} \int_{\eta_{i}}^{1} (1-s)^{2} y(s) ds    \bigg)\\
 & =&\int_{0}^{1} \bigg( G(t,s) + \frac{\alpha}{k} \int_{0}^{1} G(\tau,s) d\tau + \frac{1}{k} \sum_{i=1}^{n}\beta_{i} G(\eta_{i} ,s)\bigg) y(s) ds \\
 & =& \int_{0}^{1} H(t,s) y(s) ds .
\end{eqnarray*}
\end{proof}

\begin{lemma}\label{l2} Let $\theta\in (0,\frac{1}{2})$ be fixed. $ G(t,s)$ defined by \eqref{eq5} satisfies
\begin{itemize}
\item[(i)] $ G(t,s) \geq 0 $, for all\ $t, s \in [0,1],$
\item[(ii)] $\rho(t)e(s) \leq G(t,s)\leq e(s),$  \ for all \  $(t,s) \in [0,1] \times [0,1],$
where $ e(s)=\frac{1}{6} s(1-s)^{2},$ and
\begin{equation*}
\rho(t)=\min\{t^{3}, t^{2}(1-t)\}=\begin{cases} t^{3},& t\leq \frac{1}{2};  \\
t^{2}(1-t), & t \geq \frac{1}{2}.
\end{cases}
\end{equation*}
 \item[(iii)] $\theta^{3} e(s) \leq G(t,s)\leq e(s) ,$  \ for all \  $(t,s) \in [\theta , 1- \theta] \times [0,1].$
\end{itemize}
\end{lemma}
\begin{proof}
 See reference \emph{\cite[Lemma 2.3]{Ben}}.
\end{proof}
In the remainder of this paper, we always assume that $ k> 0$.
 \begin{lemma}\label{l4}
 Let $ 	y(t) \in C ([0,1],[0,\infty )) $ and $ \theta \in (0, \frac{1}{2})$. The unique solution  of \eqref{eq1}-\eqref{eq2} is nonnegative and satisfies
 \[ \min_{t\in [\theta, 1-\theta]} u(t) \geq \theta^{3}(1-2\theta) \|u\|.\]
 \end{lemma}
 \begin{proof}
 The positiveness of $ u(t) $ follows immediately from Lemma \ref{l1} and Lemma \ref{l2}.\\
For all $ t \in[0,1] $, we have
\begin{equation*}
 \begin{split}
  u(t)&= \int_{0}^{1} H(t,s) y(s) ds \\
  &= \int_{0}^{1} \left( G(t,s) +\frac{\alpha}{k} \int_{0}^{1}G(\tau ,s ) d\tau +\frac{1}{k} \sum_{i=1}^{n} \beta_{i} G( \eta_{i} ,s ) \right) y(s) ds \\
  &\leq \int_{0}^{1}\left( e(s) + \frac{\alpha}{k} \int_{0}^{1}e(s) d\tau  + \frac{1}{k} \sum_{i=1}^{n} \beta_{i} G( \eta_{i} ,s ) \right) y(s) ds \\
 &=  \int_{0}^{1}\left( \bigg(1+\frac{\alpha}{k} \bigg) e(s)  + \frac{1}{k} \sum_{i=1}^{n} \beta_{i} G( \eta_{i} ,s ) \right) y(s) ds .\\
 \end{split}
 \end{equation*}
 Then
 \begin{equation}\label{eq214}
 \|u\| \leq  \int_{0}^{1}\left( \bigg(1+\frac{\alpha}{k}\bigg) e(s) + \frac{1}{k} \sum_{i=1}^{n} \beta_{i} G( \eta_{i} ,s ) \right) y(s) ds.
 \end{equation}
 For $ t\in [\theta, 1-\theta] $, we have
\begin{equation}\label{eq314}
\begin{split}
 u(t)&= \int_{0}^{1} H(t,s) y(s) ds \\
 &= \int_{0}^{1} \left( G(t,s) +\frac{\alpha}{k} \int_{0}^{1}G(\tau ,s ) d\tau +\frac{1}{k} \sum_{i=1}^{n} \beta_{i} G( \eta_{i} ,s ) \right)  y(s) ds \\
&\geq \int_{0}^{1} \left(  G(t,s) +\frac{\alpha}{k} \int_{\theta}^{1-\theta} G(\tau ,s ) d\tau + \frac{1}{k} \sum_{i=1}^{n} \beta_{i} G( \eta_{i} ,s ) \right) y(s) ds \\
&\geq  \int_{0}^{1} \left( \theta^3 e(s) + \frac{\alpha}{k} \theta^3( 1- 2\theta)e(s)+\frac{1}{k}  \sum_{i=1}^{n} \beta_{i} G( \eta_{i} ,s )\right)y(s) ds \\
& \geq \theta^{3} (1-2\theta)  \int_{0}^{1} \left( \bigg(1+\frac{\alpha}{k} \bigg) e(s) + \frac{1}{k}  \sum_{i=1}^{n} \beta_{i} G( \eta_{i} ,s )\right) y(s) ds .\\
\end{split}
\end{equation}
From \eqref{eq214} and \eqref{eq314}, we obtain $$\min_{t\in [\theta, 1-\theta]} u(t) \geq  \theta^{3} (1-2\theta) \|u\|.$$
 \end{proof}
 Let $ \theta \in (0, \frac{1}{2}) $. We define the cone $$K= \left\lbrace u \in C([0,1],\ \mathbb{R}),\ u \geq 0 :  \min_{t\in [\theta ,1-\theta]} u(t) \geq  \theta^{3}(1-2\theta)  \|u\| \right\rbrace,$$
 and the operator $ T : K \rightarrow C[0,1] $ by
 \begin{equation}\label{eq16}
T u(t)= \int_{0}^{1} H(t,s) f(s , u(s)) ds ,
 \end{equation}
 where $ H(t,s) $ is defined by \eqref{eq51}.
 \begin{remark}
 By Lemma \ref{l1}, the fixed points of the operator $ T $ in $ K $ are the nonnegative solutions of the boundary value problem \eqref{eq001}-\eqref{eq002}.
 \end{remark}
 \begin{lemma}
 The operator $ T $ defined in \eqref{eq16} is completely continuous and satisfies $ TK \subset K.$
 \end{lemma}
 \begin{proof}
 From Lemma \ref{l4}, we obtain $TK \subset K $. By an application of Arzela-Ascoli theorem, $ T$ is completely continuous.
 \end{proof}

 For convenience, we introduce the following notations
 $$ f_{0} = \lim _{u\rightarrow 0^{+}} \bigg \{ \min_{0\leq t \leq 1} \frac{f(t,u)}{u}  \bigg \} ,\  f^{0} = \lim_{u\rightarrow 0^{+}} \bigg \{ \max _{0\leq t \leq 1} \frac{f(t,u)}{u} \bigg \} ,$$
 $$f_{\infty} = \lim_{u\rightarrow +\infty} \bigg \{ \min_{0\leq t \leq 1} \frac{f(t,u)}{u} \bigg \} ,\ f^{\infty} = \lim_{u\rightarrow +\infty} \bigg \{ \max _{0\leq t \leq 1} \frac{f(t,u)}{u} \bigg \} ,$$

 $$ \Psi= \theta^{6} (1-2\theta)^{2} \int_{\theta}^{1-\theta} \left( \bigg(1+\frac{\alpha}{k}\bigg) e(s) + \frac{1}{k} \sum_{i=1}^{n}\beta_{i} G( \eta_{i} ,s )\right) ds,$$
  $$ \Phi =\frac{1}{6 k},\ \Lambda_{1}= \Phi^{-1} ,\ \Lambda_{2}= \Psi^{-1}. $$

\section{Existence results}
\begin{theorem}\label{th1}
Assume that one of the following hypotheses is satisfied.
\begin{itemize}
\item[(H1)] $ f_{0} = \infty $ and $ f^{\infty} = 0$.
\item[(H2)]$ f^{0}= 0 $ and $ f_{\infty} = \infty$.
\end{itemize}
Then, the problem \eqref{eq001}-\eqref{eq002} has at least one positive solution in $ K.$
\end{theorem}
\begin{proof}
Assume that \rm{(H1)} holds.\\
Since $ f_{0}= \infty $, there exists $ \rho_1 > 0 $ such that $ f(t,u) \geq \delta u $, for all  $ 0 < u \leq \rho_1,  t \in [0, 1]$, where $ \delta > 0 $ is chosen so that
\begin{equation*}
\delta \Psi \geq 1.
\end{equation*}
Then, for $ u \in K \cap \partial \Omega_1 $ and  $ t \in [\theta,1-\theta]$ with $ \Omega_1 = \{ u \in C[0,1] : \|u \| < \rho_1 \}$, we obtain
\begin{equation}\label{eq39}
\begin{split}
T u(t)&=  \int_{0}^{1} H(t,s) y(s) ds \\
&= \int_{0}^{1} \left( G(t,s) +\frac{\alpha}{k} \int_{0}^{1}G(\tau ,s ) d\tau +\frac{1}{k} \sum_{i=1}^{n} \beta_{i} G( \eta_{i} ,s ) \right)f(s , u(s)) ds \\
&\geq \int_{\theta}^{1-\theta} \left(  G(t,s) +\frac{\alpha}{k} \int_{\theta}^{1-\theta} G(\tau ,s ) d\tau + \frac{1}{k} \sum_{i=1}^{n} \beta_{i} G( \eta_{i} ,s ) \right) f(s,u(s)) ds \\
&\geq \int_{\theta}^{1-\theta} \left(  G(t,s) +\frac{\alpha}{k} \int_{\theta}^{1-\theta} G(\tau ,s ) d\tau + \frac{1}{k} \sum_{i=1}^{n} \beta_{i} G( \eta_{i} ,s ) \right) \delta u(s) ds \\
&\geq \delta \theta^3( 1- 2\theta) \| u\|  \int_{\theta}^{1-\theta} \left( \theta^3 e(s) + \frac{\alpha}{k} \theta^3( 1- 2\theta)e(s)+\frac{1}{k}  \sum_{i=1}^{n} \beta_{i} G( \eta_{i} ,s )\right) ds  \\
&\geq \delta \theta^6( 1- 2\theta)^{2} \| u\|  \int_{\theta}^{1-\theta} \left( \bigg(1+\frac{\alpha}{k}\bigg)e(s) +\frac{1}{k}  \sum_{i=1}^{n} \beta_{i} G( \eta_{i} ,s )\right) ds  \\
& = \delta \Psi \| u \| \geq \| u \|.
\end{split}
\end{equation}
Hence, $ \| Tu \| \geq \|u\| ,\ u \in K \cap \partial \Omega_1.$ \\
On the other hand, since $ f^{\infty} = 0 $, there exists $ \widehat{{\rho}}_{2} > 0\ ( \widehat{\rho}_{2} > \rho_{1})$ such that $ f(t,u)\leq \eta u $ for all
 $ t\in [0 ,1 ] $ with $ u \geq  \widehat{\rho}_{2} $ and $ \eta > 0 $ satisfies  $$ \eta \Phi \leq 1 .$$
We consider two cases :\\
Case 1. Suppose that $f$ is bounded, then there exists $ L> 0 $ such that $ f(t,u) \leq L .$\\ Let $ \Omega_2 = \{ u \in C[0,1]: \|u\| < \rho_2\} $ with
$ \rho_2 = \max \{2 \rho_1 , L \Phi \}$.\\ If $ u \in K \cap \partial \Omega_2 $, then by Lemma \ref{l2} we have
\begin{equation}\label{eq50}
\begin{split}
T u(t)&= \int_{0}^{1} H(t,s)f(s , u(s)) ds \\
&\leq L \int_{0}^{1} \left(  e(s)  +\frac{\alpha}{k} \int_{0}^{1} e(s ) d\tau + \frac{1}{k} \sum_{i=1}^{n} \beta_{i} e(s) \right)  ds \\
&\leq L \int_{0}^{1} e(s) \left( 1 +\frac{\alpha}{k}  + \frac{1}{k} \sum_{i=1}^{n} \beta_{i}  \right)  ds \\
&= \frac{L}{k}  \int_{0}^{1} e(s) ds \\
&\leq L \Phi  \\
& \leq \rho_2 = \|u\|,
\end{split}
\end{equation}
and consequently, $ \| Tu \| \leq \| u \| ,\ u \in K \cap \partial \Omega_2.$

Case 2. If $ f $ is unbounded, then from condition \rm{(C1)}, there exists $ \sigma > 0 $ such that $ f(t,u) \leq \eta \sigma $, with $ 0 < u \leq \widehat{\rho}_{2} $ and $ t\in [0,1]. $ \\
Let  $ \Omega_2 = \{ u \in C[0,1]: \|u\| < \rho_2 \} $, where  $ \rho_2 = \max \{ \sigma , \widehat{\rho}_{2} \} $.\\
If $ u \in K \cap \partial \Omega_2 $, then we have $ f(t,u) \leq \eta \rho_{2} $, and
\begin{equation}\label{eq50512}
\begin{split}
T u(t)&= \int_{0}^{1} H(t,s) f(s , u(s)) ds \\
&\leq \int_{0}^{1} \left( e(s)  +\frac{\alpha}{k}  e(s )  + \frac{1}{k} \sum_{i=1}^{n} \beta_{i} e(s) \right) \eta \rho_{2} ds \\
&\leq \eta \rho_{2} \frac{1}{k} \int_{0}^{1} e(s) ds \\
&\leq \eta \rho_{2} \Phi \\
&\leq \rho_{2} = \| u\| .
\end{split}
\end{equation}
So,  $ \| T u \| \leq \| u \| ,\ u \in K \cap \partial \Omega_2.$\\
Therefore by Theorem \ref{T1}, $T $ has at least one fixed point, which is a positive solution of \eqref{eq001}-\eqref{eq002} such that $ \rho_1 < \| u \| < \rho_2 .$

Next, assume that \rm{(H2)} holds. \\
Since $ f^{0}= 0 $, there exists $ \rho_1 >0 $ such that $ f(t,u) \leq \epsilon u $, for all $ 0< u \leq \rho_1 , t \in [0,1]  $, where $ \epsilon > 0 $ satisfies  $$ \epsilon \Phi \leq 1 .$$ \\
Then, for $ u \in K \cap \partial \Omega_1 $ with $ \Omega_1 = \{ u \in C[0,1] : \| u \| < \rho_1 \}$, we have
\begin{equation*}
\begin{split}
T u(t)&= \int_{0}^{1}  H(t,s) f(s , u(s)) ds \\
&\leq \int_{0}^{1} \left( e(s)  +\frac{\alpha}{k}  e(s )  + \frac{1}{k} \sum_{i=1}^{n} \beta_{i} e(s) \right)\epsilon u(s)  ds \\
&\leq \frac{1}{k} \epsilon  \| u \| \int_{0}^{1} e(s) ds \\
&\leq \epsilon \Phi \| u \| \\
&\leq \| u \| .
\end{split}
\end{equation*}
Therefore, $ \| T u \| \leq \| u \| ,\ u \in K \cap \partial \Omega_1$.\\
By $ f_{\infty} = \infty$, there exists $ \widehat{\rho}_2 > 0 $ such that $ f(t,u) \geq \delta u $, for all $ u > \widehat{\rho}_2 $ and $ t \in [ \theta , 1-\theta] $, where $ \delta > 0 $ is chosen so that
$$ \delta  \Psi \geq 1 .$$ \\
Let $ \rho_2 = \max \{ 2\rho_1, \frac{\widehat{\rho}_2}{\theta^3 (1 - 2\theta)}\} $ and
$ \Omega_2 = \{ u \in C[0,1] , \| u \| < \rho_2 \} $. \\
So, for all $ u \in K \cap \partial \Omega_2 $, it is  satisfied that:
$ u(t) \geq  \widehat{\rho}_2 $, $ t \in [\theta, 1-\theta ].$ Similar to the estimates  \eqref{eq39}, we obtain
\begin{equation*}
\begin{split}
T u(t)&= \int_{0}^{1} H(t,s) f(s , u(s)) ds \\
&\geq  \delta  \Psi  \| u \| \\
&\geq \| u \|.
\end{split}
\end{equation*}
 By Theorem \ref{T1}, we deduce that there exists a positive solution of the problem \eqref{eq001}-\eqref{eq002}.
\end{proof}

\section{Multiplicity results}
\begin{theorem}\label{th2}
Assume that the following assumptions are satisfied.
\begin{itemize}
\item[(H3)] $  f_{0}= f_{\infty} = \infty $.
\item[(H4)] There exist constants $ \rho_{1} > 0 $ and $ M_{1}\in(0, \Lambda_{1}] $ such that $ f(t,u) \leq M_{1} \rho_{1},$ for $ u \in (0,\rho_{1}]$ and $ t \in [0,1]. $
\end{itemize}
Then the problem \eqref{eq001}-\eqref{eq002} has at least two positive solutions $ u_{1}$ and $  u_{2}$
such that $$ 0 < \| u_{1} \| < \rho_{1} < \| u_{2} \|. $$
\begin{proof}
First, assume that \rm{(H3)} holds. Since $ f_{0}= \infty $, then for any $ M_{*} \in [\Lambda_{2} ,\infty)$, there exists $ \rho_{*}\in (0,\rho_{1}) $  such that $ f(t,u) \geq M_{*} u $, for all $ t\in [\theta , 1- \theta] $ and $ 0< u \leq \rho_{*} $. \\
 Set $\Omega_{\rho_{*}}= \{u \in C[0,1]: \| u \| < \rho_{*} \}.$ By using Lemma \ref{l2}, for $ u \in K \cap \partial\Omega_{\rho_{*}} $ and $ t \in [\theta, 1-\theta] $, we have
\begin{equation}\label{eq40}
\begin{split}
T u(t)&= \int_{0}^{1}H(t,s) f(s , u(s)) ds \\
&\geq \int_{\theta}^{1-\theta} \left(  G(t,s) +\frac{\alpha}{k} \int_{\theta}^{1-\theta} G(\tau ,s ) d\tau + \frac{1}{k} \sum_{i=1}^{n} \beta_{i} G( \eta_{i} ,s ) \right) f(s,u(s)) ds \\
&\geq \int_{\theta}^{1-\theta} \left(  G(t,s) +\frac{\alpha}{k} \int_{\theta}^{1-\theta} G(\tau ,s ) d\tau + \frac{1}{k} \sum_{i=1}^{n} \beta_{i} G( \eta_{i} ,s ) \right) M_{*} u(s) ds \\
&\geq M_{*} \theta^6( 1- 2\theta)^{2} \left[ \int_{\theta}^{1-\theta} \left( \bigg(1+\frac{\alpha}{k}\bigg) e(s) + \frac{1}{k}  \sum_{i=1}^{n} \beta_{i} G( \eta_{i} ,s )\right) ds \right]  \rho_{*} \\
& = M_{*}  \Lambda^{-1}_{2} \rho_{*} \\
& \geq \Lambda_{2}  \Lambda^{-1}_{2} \rho_{*} \\
& = \| u \|,
\end{split}
\end{equation}
which means that
 \begin{equation}\label{eq41}
  \| Tu \| \geq \| u \|,\ u \in K \cap  \partial \Omega_{\rho_{*}} .
\end{equation}
 On the other hand, since $ f_{\infty} = \infty $, then for any $ M^{*} \in[\Lambda_{2} , \infty)$, there exists $ \bar{{\rho}^{*}} > \rho_{1} $ such that $ f(t,u) \geq M^{*} u $, for all $ t \in [\theta , 1- \theta]$ and $ u \geq \bar{\rho^{*}} $. \\
 Let $ \rho^{*} \geq  \frac{\bar{\rho^{*}}}{\theta^3 (1-2\theta)} $ and $ \Omega_{\rho^{*}}=\{ u \in C[0,1] : \| u \| < \rho^{*}\}$. For all $ u \in K \cap \partial \Omega_{\rho^{*}} $, we have that $u(t) \geq  \bar{\rho^{*}},\ t \in [\theta,1-\theta]$.\\
  Hence, for $ t\in [\theta, 1-\theta] $, we get
\begin{equation}\label{eq159951}
\begin{split}
T u(t)&= \int_{0}^{1} H(t,s) f(s , u(s)) ds \\
&\geq \int_{\theta}^{1-\theta} \left(  G(t,s) +\frac{\alpha}{k} \int_{0}^{1} G(\tau ,s ) d\tau + \frac{1}{k} \sum_{i=1}^{n} \beta_{i} G( \eta_{i} ,s ) \right) f(s,u(s)) ds \\
&\geq \int_{\theta}^{1-\theta} \left(  G(t,s) +\frac{\alpha}{k} \int_{\theta}^{1-\theta} G(\tau ,s ) d\tau + \frac{1}{k} \sum_{i=1}^{n} \beta_{i} G( \eta_{i} ,s ) \right) M^{*} u(s) ds \\
& \geq \rho^{*} M^{*} \Lambda^{-1}_{2} \\
& \geq \rho^{*} \Lambda_{2} \Lambda^{-1}_{2} \\
& = \| u\|.
 \end{split}
\end{equation}
 Therefore
 \begin{equation}\label{eq42}
  \| Tu \| \geq \| u \| ,\ u \in K \cap \partial \Omega_{\rho^{*}}.
\end{equation}
Finally, set $ \Omega_{\rho_{1}} = \{u \in C[0,1] : \| u \| < \rho_{1} \} $. Then for any $ u \in K \cap \partial\Omega_{\rho_{1}} $, we get from \rm{(H4)} that $ f(t,u) \leq  M_{1} \rho_{1} $ for all $ t \in [0, 1] $, and similar to the estimates   \eqref{eq50}, we obtain
 \begin{equation}\label{eq9191}
 \begin{split}
 T u(t)&= \int_{0}^{1} \left( G(t,s) +\frac{\alpha}{k} \int_{0}^{1}G(\tau ,s ) d\tau +\frac{1}{k} \sum_{i=1}^{n} \beta_{i} G( \eta_{i} ,s ) \right)f(s , u(s)) ds \\
 &\leq \int_{0}^{1} e(s) \left(1 +\frac{\alpha}{k}   + \frac{1}{k} \sum_{i=1}^{n} \beta_{i}  \right)M_{1} \rho_{1} ds \\
 &\leq \Lambda_{1} \rho_{1}\frac{1}{k} \int_{0}^{1} e(s)   ds \\
 & \leq \Lambda_{1} \Lambda^{-1}_{1} \rho_{1} \\
 &= \| u \|,
 \end{split}
 \end{equation}
  which yields
  \begin{equation}\label{eq44}
   \| Tu \| \leq \| u \| ,\ u \in K \cap \partial\Omega_{\rho_{1}}.
  \end{equation}
Hence, since $ \rho_{*} < \rho_{1} < \rho^{*}$ and from \eqref{eq41},\eqref{eq42},\eqref{eq44}, it follows from Theorem \ref{T1} that $ T $ has a fixed point $ u_{1} $ in $ K \cap (\bar{\Omega}_{\rho_{1}}\setminus \Omega_{\rho_{*}})$ and a fixed point $ u_{2} $  in $  K \cap (\bar{\Omega}_{\rho^{*}}\setminus \Omega_{\rho_{1}}) .$\\
Both one positive solutions of the problem \eqref{eq001}-\eqref{eq002} and $ 0 < \| u_{1}\| < \rho_{1} < \| u_{2} \| .$
\end{proof}
\end{theorem}

\begin{theorem}\label{th3}
Assume that the following assumptions are satisfied.
\begin{itemize}
\item[(H5)] $ f^{0}= f^{\infty} = 0 $.
\item[(H6)] There exist constants $ \rho_{2} >0 $ and $ M_{2}\in [\Lambda_{2},\infty)$ such that $ f(t,u) \geq M_{2} \rho_{2}$, for $ u \in [\theta^{3}(1-2\theta)\rho_{2},\rho_{2} ]$ and $ t \in [\theta,1-\theta]$.
\end{itemize}
Then the problem \eqref{eq001}-\eqref{eq002} has at least two positive solutions $ u_{1}$ and $  u_{2}$
such that $$ 0 < \| u_{1} \| < \rho_{2} < \| u_{2} \| .$$
\begin{proof}
Assume that \rm{(H5)} holds. Firstly, since $ f^{0} =0 $, for any $ \epsilon \in (0,\Lambda_{1}]$, there exists $ \rho_{*} \in (0,\rho_{2})  $ such that $ f(t,u) \leq \epsilon u $, for all $ t\in [0, 1 ] $ where $ 0< u \leq \rho_{*}.$\\
 Then, for $ u \in K \cap \partial\Omega_{\rho_{*}}$ with $ \Omega_{\rho_{*}}= \{ u \in C[0,1] : \| u \| < \rho_{*} \} $, we have
\begin{equation*}
\begin{split}
T u(t)&= \int_{0}^{1} H(t,s) f(s , u(s)) ds \\
&\leq \int_{0}^{1} \left( e(s)  +\frac{\alpha}{k}  e(s )  + \frac{1}{k} \sum_{i=1}^{n} \beta_{i} e(s) \right)f(s, u(s) )  ds \\
&\leq \int_{0}^{1}e(s) \left(1  +\frac{\alpha}{k}  + \frac{1}{k} \sum_{i=1}^{n} \beta_{i} \right)\epsilon u(s)   ds \\
&\leq  \epsilon \rho_{*} \frac{1}{k} \int_{0}^{1} e(s) ds\\
&\leq \epsilon \Lambda^{-1}_{1} \rho_{*}\\
&\leq \rho_{*}=\| u \|.
\end{split}
\end{equation*}
 Therefore
\begin{equation}\label{eq47}
 \| Tu \| \leq \| u \|, \ u \in K \cap \partial\Omega_{\rho_{*}}.
\end{equation}
\\
Secondly, in view of $ f^{\infty} = 0 $, for any $ \epsilon_{1} \in (0,\Lambda_{1}]$, there exists $ \widetilde{\rho}> \rho_{2} $ such that $ f(t,u) \leq \epsilon_{1} u $, for all $ t \in [0, 1] $ with $ u \geq \widetilde{\rho}.$

We consider two cases : \\
Case 1. Suppose that $ f $ is bounded. Let $ L > 0 $ be such that $ f(t,u) \leq L$, for all $u\in [0,\infty)$ and $t\in [0,1]$.\\
Taking $ \rho^{*} \geq \max \{\widetilde{\rho}, \frac{L}{\epsilon_{1}}\}.$
For $u\in K$ with $\|u\|=\rho^{*}$, we have
\begin{equation*}
\begin{split}
T u(t)&= \int_{0}^{1} \left( G(t,s) +\frac{\alpha}{k} \int_{0}^{1}G(\tau ,s ) d\tau +\frac{1}{k} \sum_{i=1}^{n} \beta_{i} G( \eta_{i} ,s ) \right)f(s , u(s)) ds \\
&\leq L \Phi \\
&\leq \rho^{*} \epsilon_{1}\Lambda_{1}^{-1} \\
&\leq \rho^{*} =\| u \|,
\end{split}
\end{equation*}
and consequently
\begin{equation}\label{eq9876}
 \| Tu \| \leq \| u \| ,\ u \in K \cap \partial \Omega_{\rho^{*}}.
\end{equation}
Case 2. Suppose that $ f $ is unbounded, then from condition \rm{(C1)}, there exists $ \sigma > 0 $ such that $ f(t,u) \leq \epsilon_{1} \sigma $, with $ 0 \leq u \leq \widetilde{\rho} $, and $t\in[0,1]$.\\
For $u\in K$ with $\|u\|=\rho^{*}$, where $ \rho^{*} \geq \max \{ \sigma , \widetilde{\rho}\} $, we obtain
\begin{equation*}
\begin{split}
T u(t)&= \int_{0}^{1} H(t,s) f(s , u(s)) ds \\
&\leq \int_{0}^{1} \left( e(s)  +\frac{\alpha}{k} e(s )+ \frac{1}{k} \sum_{i=1}^{n} \beta_{i} e(s) \right)f(s, u(s) )  ds \\
&\leq \int_{0}^{1}e(s) \left(1  +\frac{\alpha}{k}  + \frac{1}{k} \sum_{i=1}^{n} \beta_{i} \right)\epsilon_{1} \rho^{*} ds \\
&\leq  \epsilon_{1} \rho^{*} \Lambda^{-1}_{1}\\
&\leq \rho^{*}=\| u \|.
\end{split}
\end{equation*}
We conclude that
\begin{equation}\label{eq8765}
 \| Tu \| \leq \| u \| ,\ u \in K \cap \partial \Omega_{\rho^{*}}.
\end{equation}
Hence, in either case, we always may set $ \Omega_{\rho^{*}} = \{u \in C[0,1] : \| u \| < \rho^{*} \} $ such that $\| Tu \| \leq \| u \|$,\ for $\ u \in K \cap \partial \Omega_{\rho^{*}}.$

Now, set $ \Omega_{\rho_{2}} = \{u \in C[0,1] : \| u \| < \rho_{2} \} $. Then for any $ u \in K \cap \partial\Omega_{\rho_{2}} $, we get from \rm{(H6)} that there exists $M_{2} \in [\Lambda_{2},\infty)$ such that $ f(t,u) \geq M_{2} \rho_{2} $ for all  $ t \in [\theta , 1- \theta ]$, and $ u \in [\theta^{3}( 1-2\theta) \rho_{2},\rho_{2}].$ Similar to the  estimates of \eqref{eq159951}, we get
\begin{equation}\label{eq8181}
\begin{split}
T u(t) &\geq M_{2} \theta^{3}( 1- 2\theta) \rho_{2} \int_{\theta}^{1-\theta} \left( \bigg(1+\frac{\alpha}{k}\bigg) e(s) + \frac{1}{k}  \sum_{i=1}^{n} \beta_{i} G( \eta_{i} ,s )\right) ds  \\
 &\geq M_{2} \rho_{2} \theta^6( 1- 2\theta)^{2} \int_{\theta}^{1-\theta} \left( \bigg(1+\frac{\alpha}{k}\bigg) e(s) + \frac{1}{k}  \sum_{i=1}^{n} \beta_{i} G( \eta_{i} ,s )\right) ds  \\
& = M_{2} \Lambda_{2}^{-1} \rho_{2}  \\
& \geq \rho_{2}=\|u\|.
\end{split}
\end{equation}
 Then
\begin{equation}\label{eq8888}
 \|Tu\| \geq \| u \| ,\ u \in K \cap \partial\Omega_{\rho_{2}}.
\end{equation}
Hence, from \eqref{eq47},\eqref{eq9876},\eqref{eq8765} and \eqref{eq8888}, it follows from Theorem \ref{T1} that there exist at least two positive solutions $ u_{1} $ in $ K \cap (\bar{\Omega}_{\rho_{2}} \setminus \Omega_{\rho_{*}})$  and $ u_{2} $  in $  K \cap (\bar{\Omega}_{\rho^{*}} \setminus \Omega_{\rho_{2}})$ of the problem \eqref{eq001}-\eqref{eq002} such that $ 0 < \| u_{1}\| < \rho_{2} < \| u_{2} \| .$
\end{proof}
\end{theorem}

 \section{Examples}
\begin{exmp}
Consider the boundary value problem

\begin{equation}\label{eq5.1}
       \begin{cases}u^{\prime \prime \prime \prime} (t) + t+ |\cos u |=0,\ 0 < t < 1,\\
    u^{\prime}(0) = u^{\prime}(1) = u^{\prime \prime}(0) =  0,\\
     u(0) = \frac{1}{3} \int_{0}^{1} u(s) ds + \frac{1}{7} u(\frac{7}{15}) + \frac{1}{4} u(\frac{2}{3}) + \frac{3}{84} u(\frac{11}{13}),
       \end{cases}
       \end{equation}
where  $ f(t,u) =t+ |\cos u | , \alpha = \frac{1}{3},\beta_{1}=\frac{1}{7},\beta_{2}=\frac{1}{4},\beta_{3}=\frac{3}{84},\eta_{1}=\frac{7}{15}, \eta_{2}=\frac{2}{3},$ and $ \eta_{3}=\frac{11}{13}$.\\
 We have $ k= 1-( \frac{1}{3}+ \frac{1}{7}+  \frac{1}{4} +  \frac{3}{84} ) = \frac{5}{21}> 0 ,f_{0} = \infty , f^{\infty} = 0 .$\\ Then, by \rm{(H1)} of Theorem \ref{th1} , the problem \eqref{eq5.1} has at least one positive solution.
\end{exmp}

\begin{exmp}
As a second example we consider the following boundary value problem

\begin{equation}\label{eq5.2}
\begin{cases}
u^{\prime \prime \prime \prime} (t)+ u^{2}  e^{u} \ln (1+t+u),\ 0 < t < 1,\\
 u^{\prime}(0) = u^{\prime}(1) = u^{\prime \prime}(0) =  0,\\
 u(0) = \frac{1}{4} \int_{0}^{1} u(s) ds + \frac{1}{12} u(\frac{1}{8}) + \frac{1}{6} u(\frac{1}{4}),
  \end{cases}
\end{equation}
where  $ f(t,u) = u^{2} e^{u} \ln (1+t+u) \geq 0 , \alpha = \frac{1}{4}, \beta_{1}=\frac{1}{12}, \beta_{2}=\frac{1}{6},\eta_{1}=\frac{1}{8},$ and $ \eta_{2}=\frac{1}{4}$. \\
We have $ k= 1-(\frac{1}{4}+ \frac{1}{12} + \frac{1}{6} ) = \frac{1}{2} > 0 , f^{0}  = 0 , f_{\infty}= \infty $.\\ So, by  \rm{(H2)} of Theorem \ref{th1}, the problem \eqref{eq5.2} has at least one positive solution.
\end{exmp}

\begin{exmp}
Let the following boundary value problem

\begin{equation}\label{eq5.3}
\begin{cases}
u^{\prime \prime \prime \prime} (t) + (1+t) e^{u}=0,\ 0 < t < 1,\\
 u^{\prime}(0) = u^{\prime}(1) = u^{\prime \prime}(0) = 0,\\
 u(0) = \frac{1}{30} \int_{0}^{1} u(s) ds + \frac{1}{60} u(\frac{1}{4}) + \frac{1}{120} u(\frac{1}{3}) + \frac{1}{240} u(\frac{1}{2}),
 \end{cases}
\end{equation}
where $ f(t,u) = (1+t) e^{u}, \alpha = \frac{1}{30}, \beta_{1}= \frac{1}{60}, \beta_{2}= \frac{1}{120}, \beta_{3}=\frac{1}{240}, \eta_{1}= \frac{1}{4}, \eta_{2}=\frac{1}{3},$ and $ \eta_{3}=\frac{1}{2}.$

Then $ f_{0} =f_{\infty} = \infty , k= 1-(\frac{1}{30} + \frac{1}{60} + \frac{1}{120} + \frac{1}{240} )= \frac{15}{16}.$ \\
On the other hand, choosing $ \rho_{1}= 1$ and $ M_{1} =\Lambda_{1}$. Then  $ f(t,u) \leq 2 e $ , for $(t,u) \in [0,1]\times (0,\rho_{1}] $ and $ \Lambda_{1} = 6k = \frac{45}{8} = 5,625.$ So
  $$ f(t,u) \leq 2 e \leq 5,625= M_{1} \rho_{1}.$$

 By Theorem \ref{th2}, the problem \eqref{eq5.3} has at least two positive solutions.
\end{exmp}

\begin{exmp}
Consider the following boundary value problem

\begin{equation}\label{eq5.4}
\begin{cases}
u^{\prime \prime \prime \prime} (t)+6528 \times 10^{9} u^{2}e^{1-u}=0,\ 0 < t < 1,\\
u^{\prime}(0) = u^{\prime}(1) = u^{\prime \prime}(0) =  0,\\
u(0) = \frac{1}{10} \int_{0}^{1} u(s) ds + \frac{1}{20} u(\frac{1}{2}),
 \end{cases}
\end{equation}
where  $ f(t,u) = f(u) = 6528 \times 10^{9} u^{2}e^{1-u}, \alpha = \frac{1}{10}, \beta_{1}=\beta= \frac{1}{20}$ and $ \eta_{1}=\eta=\frac{1}{2}. $ Then  $ f^{0} =f^{\infty} = 0,\ k= \frac{17}{20} > 0,\ (1+ \frac{\alpha}{k})= \frac{19}{17},\ \frac{\beta}{k}= \frac{1}{17}.$
And
$$ \Psi= \theta^{6} (1-2\theta)^{2} \int_{\theta}^{1-\theta} \left( \bigg(1+\frac{\alpha}{k}\bigg) e(s) + \frac{\beta}{k}  G\bigg( \frac{1}{2} ,s \bigg)\right) ds .$$
So
\begin{equation}
\begin{split} \label{eq987963}
\Psi= &\theta^{6} (1-2\theta)^{2} \bigg[\int_{\theta}^{1-\theta} \frac{19}{17} \frac{1}{6} s(1-s)^{2} ds + \frac{1}{17} \frac{1}{6}  \int_{\theta}^{\frac{1}{2}} \frac{1}{8} (1-s)^{2} ds\\
& - \frac{1}{17} \frac{1}{6}  \int_{\theta}^{\frac{1}{2}}  \bigg(\frac{1}{2}-s\bigg)^{3} ds +  \frac{1}{17} \frac{1}{6}  \int_{\frac{1}{2}}^{1-\theta} \frac{1}{8} (1-s)^{2} ds \bigg]\\
= & \frac{\theta^{6} (1-2\theta)^{2}}{102} \bigg[19 \Psi_{1}+ \frac{1}{8} \Psi_{2}^{1} - \Psi_{2}^{2} + \frac{1}{8}\Psi_{3}
 \bigg ],
\end{split}
\end{equation}
with
\begin{align*}
\Psi_{1} &= \int_{\theta}^{1-\theta} s(1-s)^{2} ds = \frac{1}{6}(1-2\theta) \Big(\frac{1}{2}+\theta-\theta^{2}\Big) ,\\
\Psi_{2}^{1} &= \int_{\theta}^{\frac{1}{2}} (1-s)^{2}ds = \frac{1}{6}(1-2\theta) \Big(\frac{7}{4}-\frac{5}{2}\theta+\theta^{2}\Big),\\
\Psi_{2}^{2} &= \int_{\theta}^{\frac{1}{2}} \Big(\frac{1}{2}-s\Big)^{3}ds = \frac{1}{64}(1-2\theta)^{4},\\
\Psi_{3} &= \int_{\frac{1}{2}}^{1-\theta} (1-s)^{2}ds =\frac{1}{6}(1-2\theta) \Big(\frac{1}{4}+\frac{1}{2}\theta+\theta^{2}\Big) .
\end{align*}
\begin{equation*}
\Psi=  \frac{1}{6528}\theta^{6}{(1-2\theta)^{3}}(103+206{\theta}-212{\theta^{2}}+8{\theta^{3}}).
\end{equation*}

So
  $$ \Lambda_{2} = \frac{6528}{\theta^{6}{(1-2\theta)^{3}}(103+206{\theta}-212{\theta^{2}}+8{\theta^{3}})}.$$

 On the other hand, let us choose $ \rho_{2}= 1$ and $  M_{2} =\Lambda_{2}$. Then\\
   $ f(t,u)= f(u)  \geq  6528 \times 10^{9} \theta^{6}(1-2\theta)^{2} $, for $(t,u) \in [\theta, 1-\theta]\times [\theta^{3} (1-2\theta) \rho_{2},\rho_{2}].$\\
 So
 \begin{equation*}
 f(t,u)\geq 10^{9} \theta^{12}(1-2\theta)^{5} (103+206{\theta}-212{\theta^{2}}+8{\theta^{3}}) \Lambda_{2} .
 \end{equation*}

 Using the Mathematica software, we easily check that
  \begin{equation*}
 10^{9} \theta^{12}(1-2\theta)^{5} (103+206{\theta}-212{\theta^{2}}+8{\theta^{3}})  \geq 1, \ \text{for all}\ \theta \in
  \Big[\frac{17}{125},\frac{12}{25}\Big],
  \end{equation*}
 and consequently  $$ f(t,u) \geq \Lambda_{2} = M_{2}.$$

 By Theorem \ref{th3}, the problem \eqref{eq5.4} has at least two positive solutions.
\end{exmp}

\end{document}